\documentclass{IEEEojcsys}

\usepackage[colorlinks,urlcolor=blue,linkcolor=blue,citecolor=blue]{hyperref}

\usepackage{color,array}
\usepackage{amsmath}
\usepackage{amssymb}
\usepackage{mathtools}
\usepackage{arydshln} 
\usepackage{mdframed} 
\usepackage{graphicx}
\usepackage{amsthm}
\usepackage{tikz}
\usepackage[noadjust]{cite}
\usepackage{pgfplots}
\pgfplotsset{compat=newest}
\pgfplotsset{plot coordinates/math parser=false} 

\makeatletter
\def\endthebibliography{%
	\def\@noitemerr{\@latex@warning{Empty `thebibliography' environment}}%
	\endlist
}
\makeatother 

\theoremstyle{plain}
\newtheorem{theorem}{Theorem}
\newtheorem{lemma}{Lemma}
\theoremstyle{definition}
\newtheorem{assumption}{Assumption}
\newtheorem{definition}{Definition}
\theoremstyle{remark}
\newtheorem{remark}{Remark}
\newcommand{\norm}[1]{\left\lVert#1\right\rVert}
\newcommand\numberthis{\addtocounter{equation}{1}\tag{\theequation}}
\renewcommand{\refeq}[1]{\overset{#1}{=}}
\newcommand{\refleq}[1]{\overset{#1}{\leq}}

\begin{document}


\title{Online convex optimization for data-driven control of dynamical systems} 

\author{M. Nonhoff\affilmark{1}}

\author{M. A. Müller\affilmark{1}  (Senior Member, IEEE)}

\affil{Leibniz University Hannover, Institute of Automatic Control, Hannover, Germany} 

\corresp{CORRESPONDING AUTHOR: M. Nonhoff (e-mail: \href{mailto:nonhoff@irt.uni-hannover.de}{nonhoff@irt.uni-hannover.de})}
\authornote{This work was supported by the Deutsche Forschungsgemeinschaft (DFG, German Research Foundation) - 505182457.}

\begin{abstract}
We propose an algorithm based on online convex optimization for controlling discrete-time linear dynamical systems. The algorithm is data-driven, i.e., does not require a model of the system, and is able to handle a priori unknown and time-varying cost functions. To this end, we make use of a single persistently exciting input-output sequence of the system and results from behavioral systems theory which enable it to handle unknown linear time-invariant systems. Moreover, we consider noisy output feedback instead of full state measurements and allow general economic cost functions. Our analysis of the closed loop reveals that the algorithm is able to achieve sublinear regret, where the measurement noise only adds an additional constant term to the regret upper bound. In order to do so, we derive a data-driven characterization of the steady-state manifold of an unknown system. Moreover, our algorithm is able to asymptotically exactly estimate the measurement noise. The effectiveness and applicational aspects of the proposed method are illustrated by means of a detailed simulation example in thermal control.
\end{abstract}


\maketitle

\section{INTRODUCTION}
This paper considers the problem of controlling an unknown linear time-invariant (LTI) system subject to time-varying and a priori unknown convex cost functions. In particular, we aim to minimize the accumulated cost obtained by our proposed algorithm in closed loop with the unknown system. The main difficulty arises from the fact that the cost functions are time-varying and a priori unknown, i.e., the cost function $L_t$ at time $t$ is only revealed to us at time step $t+1$. These kind of problems commonly arise in practice, e.g., in power grids due to a priori unknown renewable energy generation and unknown energy consumption \cite{Picallo2020}, in data center cooling \cite{Lazic2018}, or in robotics \cite{Zheng2020}. Our approach is inspired by online convex optimization (OCO) \cite{Hazan2016,Simonetto2020}, an online variant of classical numerical optimization. Whereas the classical OCO literature does not consider underlying dynamical systems, it has gained significant interest recently for solving optimal control tasks. Its main advantages include its ability to handle a priori unknown and time-varying cost functions, low computational complexity, and its ability to take constraints on the state and the input of the system into account. OCO-based algorithms have been proposed to control linear dynamical systems \cite{Nonhoff2020,Li2019} subject to process noise \cite{Agarwal2019,Hazan2020}, constraints \cite{Nonhoff2021,Li2021}, or output feedback \cite{Simchowitz2020}.

Most of the existing OCO-based algorithms in the literature discussed above depend crucially on model knowledge of the system. However, obtaining such a model can be difficult or expensive in certain applications. Hence, in recent years, direct data-based control approaches have received a considerable amount of attention, compare, e.g., \cite{Markovsky2021}. In this work, we employ a result from behavioral systems theory. The so-called fundamental lemma shows that a Hankel matrix consisting of a single persistently exciting input-output trajectory spans the whole vector space of all possible input-output trajectories of an LTI system \cite{Willems2005}. This result has recently drawn significant attention and has been applied to solve a variety of control problems, e.g., model predictive control (MPC) \cite{Coulson2019,Berberich2021}, state- and output-feedback design \cite{Berberich2020ACC,Xue2021,DePersis2020,Berberich2020arXiv,vanWaarde2022}, and output matching \cite{Markovsky2008}. We combine the fundamental lemma with OCO in order to control dynamical systems subject to time-varying cost functions, where neither the system nor the cost functions are known to the algorithm.

Another closely related line of research is so-called optimal steady-state (OSS) control. Therein, a system is controlled to the solution of a (possibly time-varying) optimization problem by applying gradient-based feedback and, typically, asymptotic guarantees in the form of stability of the overall system are derived \cite{Menta2018,Colombino2020}. Again, the main focus in the literature is on model-based control with process noise and output feedback \cite{Lawrence2018,Zheng2020, Cothren2022}. In~\cite{He2022}, a data-driven method for regulating the output of a general nonlinear system, subject to a constant disturbance, to the optimal steady state of a constant cost function is proposed. In particular, the authors leverage a result from zeroth order optimization in order to avoid requiring model knowledge of the controlled system. However, performance is only analyzed in terms of the second moments of the gradients of a smooth approximation of the cost function. Most relevant to this work is \cite{Bianchin2021}, where output feedback and unknown systems subject to disturbances are treated by application of the fundamental lemma. To this end, a steady-state map between the input to and the output of the unknown systems is estimated using only measured data. However, the cost functions are assumed to be constant and time-variability of the optimization problem is only introduced via time-varying process noise. Moreover, analysis of the closed loop's transient behavior is limited to analysis of contraction with respect to the optimal steady state, but does not consider the transient cost in terms of regret analysis.

The contribution of this work is fivefold. First, we consider an unknown system by leveraging results from data-driven control. Compared to alternative approaches in the literature, we thereby remove the need of a (set-based) model description and of an online estimation process. Second, we extend our previous results from OCO-based control \cite{Nonhoff2020,Nonhoff2021} to the case of output feedback instead of full state measurements, which requires considerable adjustments in algorithm design and analysis techniques. Third, we consider noise in the measurement process. In the relevant literature, e.g., \cite{Bianchin2021,Hazan2020,Li2021}, the main research focus is on systems subject to process noise, which is typically handled by estimating the process noise using exact measurements and model knowledge. We instead consider only noisy measurements in our theoretical work and leave the combination of both, process and measurement noise, as an interesting topic for future research. We do, however, consider both types of noise in our simulation example. Fourth, we generalize previous work \cite{Nonhoff2020,Nonhoff2021} by considering the practically relevant case of \textit{economic} cost functions, i.e., the minimum of the cost functions at each time step need \textit{not} be a steady state of the system. Finally, we derive a new data-driven characterization of the steady-state manifold of an LTI system by leveraging the fundamental lemma. As a main result, our analysis reveals that our proposed algorithm enjoys sublinear regret without access to a system model or exact measurements.

This paper is organized as follows. In Section~\ref{sec:setting}, we present the basic notions necessary in our work and discuss the problem of interest. Section~\ref{sec:algorithm} introduces and illustrates our proposed algorithm. In Section~\ref{sec:theoretical_results}, we discuss our theoretical findings, in particular a regret analysis of the closed loop and asymptotic convergence of the measurement error estimates. A numerical simulation example, namely a thermal control problem, illustrates the closed-loop performance and applicational aspects of our algorithm in Section~\ref{sec:numerical_example}. Section~\ref{sec:conclusion} concludes the paper.

We close this section by noting that a preliminary version of parts of this paper was presented at the 2021 60th IEEE Conference on Decision and Control (CDC) \cite{Nonhoff2021CDC}. This work extends the previously presented results in three directions. First, we consider measurement noise in this work and study its effect on the derived regret bound, which requires adaptations in both algorithm design and theoretical analysis. We show that measurement noise only leads to an additional \textit{constant} term in the regret bound compared to our previous work. Second, we generalize our work to consider economic cost functions, as discussed above. Third, we remove restrictive assumptions on the steady-state manifold, compare~\cite[Assumption~1]{Nonhoff2021CDC}, in order to be able to control a wider class of systems. Moreover, we include a detailed simulation example to illustrate the applicability of our proposed algorithm.

\textit{Notation:} We denote the set of integer numbers in the interval $[a,b]$ and the set of integer numbers greater than or equal to zero by $\mathbb I_{[a,b]}$ and $\mathbb I_{\geq 0}$, respectively. For a vector \mbox{$x \in \mathbb R^n$}, $\norm{x}$ is the euclidean norm and for a matrix $A \in \mathbb R^{n\times m}$ the corresponding induced matrix 2-norm is $\norm{A}$, whereas its Moore-Penrose-Pseudoinverse is denoted by $A^\dagger$. The identity matrix of size $n\times n$ is given by $I_n$, $1_n \in \mathbb R^n$ denotes the vector of all ones, and $0_n \in \mathbb R^n$ is the vector of all zeros. A sequence $\{z_k\}_{k=0}^{N-1}$, $z_k \in \mathbb R^n$, induces the Hankel matrix of depth $L$
\[
	H_L(z) = \begin{bmatrix} z_0 & z_1 & \dots & z_{N-L} \\ z_1 & z_2 & \dots & z_{N-L+1} \\ \vdots & \vdots & \ddots & \vdots \\ z_{L-1} & z_{L} & \dots & z_{N-1} \end{bmatrix} = \begin{bmatrix} H_L^1(z) \\ H_L^2(z) \\ \vdots \\ H_L^L(z) \end{bmatrix}.
\]
We denote a matrix containing a subset of block rows of $H_L(z)$ by
\[
	H_L^{a:b}(z) = \begin{bmatrix} H_L^a(z) \\ \vdots \\ H_L^b(z) \end{bmatrix}.
\]
With a slight abuse of notation, we write $z$ for the sequence itself as well as for the stacked vector of all its components. We denote by $z_{[a:b]} = \begin{bmatrix} z_a^\top & \dots & z_b^\top \end{bmatrix}^\top$ the stacked vector of a subset of its components. The shift operator $\sigma$ is defined by $\sigma z = \begin{bmatrix} z_1^\top & \dots & z_{N-1}^\top \end{bmatrix}^\top$. For matrices $A$ and $B$, $A \otimes B$ denotes the Kronecker product.

\section{SETTING} \label{sec:setting}
We consider linear time-invariant (LTI) systems of the form
\begin{align}
	\begin{split}
		x_{t+1} &= Ax_t + Bu_t \\
		y_t &= Cx_t + Du_t \\
		\tilde y_t &= y_t + e_t
	\end{split} \label{eq:system_ss}
\end{align}
where $x_t \in \mathbb R^n$ is the system state, $u_t \in \mathbb R^m$ is the system input, $y_t \in \mathbb R^p$ is the true system output, $\tilde y_t \in \mathbb R^p$ is the measured system output, and $e_t \in \mathbb R^p$ denotes measurement noise at time instance $t$. We denote by $z_t = \begin{bmatrix} u_t^\top & y_t^\top \end{bmatrix}^\top$ the stacked input-output pair at time $t$. The system matrices $(A,B,C,D)$ as well as the noise $e_t$ are unknown and only measurements of $u_t$ and $\tilde y_t$ are available to us. We do not impose any assumptions on the measurement noise~$e$. We make the following assumptions on system~\eqref{eq:system_ss}.
\begin{assumption} \label{assump:system_properties} The matrix $A$ is Schur stable, the pair $(A,B)$ is controllable, and the pair $(A,C)$ is observable.
\end{assumption}
Controllability and Observability are standard assumptions in the literature \cite{Bianchin2021}. Compared to~\cite{Nonhoff2021CDC}, we only consider stable systems because of the additional measurement noise. In this setting, we can estimate the measurement error asymptotically exactly, if the system is stable (compare Lemma~\ref{lem:error_estimate_convergence} below). If the system is not stable, data-based techniques from, e.g., \cite{DePersis2020,Berberich2020arXiv} can be used to stabilize the (unknown) system. Our algorithm can then be applied to the prestabilized system. However, some of our theoretical guarantees deteriorate for this approach, compare Remark~\ref{rem:unstable_systems} for more details.

Our goal is to solve the optimal control problem
\[
	\min_{u} \quad \sum_{t=0}^T L_t(u_t,y_t) \qquad 
	\text{s.t.} \quad \eqref{eq:system_ss}, \numberthis \label{eq:OCP}
\]
here the main difficulty arises from the fact that the time-varying cost functions $L_t: \mathbb R^m \times \mathbb R^p \rightarrow \mathbb R$ are a priori unknown. Specifically, we want to find a controller that computes an input $u_t$ at every time instance $t$ which is applied to system \eqref{eq:system_ss} and yields performance close to the solution of~\eqref{eq:OCP}. Only after $u_t$ is applied to system~\eqref{eq:system_ss}, the cost function $L_t$ is revealed, i.e., $u_t$ is computed by the algorithm without knowledge of the current cost function. Then, we measure the noisy output $\tilde y_t$ and move to the next time step. As standard in OCO, we do not attempt to solve~\eqref{eq:OCP} directly at each time step \cite{Hazan2016,Simonetto2020}. Since the cost functions are a priori unknown, optimization would have to be carried out based on the last known cost function $L_{t-1}$. Then, open-loop optimization will in general not improve the closed-loop performance, due to the time-varying nature of the cost functions. Therefore, we aim to design a computationally efficient algorithm, instead of solving a (potentially large-scale) optimization problem at each step. We denote the solution to \eqref{eq:OCP} in hindsight, i.e., the solution when knowing all cost functions, by $u^*=\{u^*_t\}_{t=0}^T$ and the corresponding system output by $y^*=\{y^*_t\}_{t=0}^T$. As common in OCO, we consider smooth convex cost functions as specified in Assumption~\ref{assump:cost_fcn}.
\begin{assumption} \label{assump:cost_fcn} The cost functions $L_t(z)$ are
\begin{itemize}
	\item $\alpha_z$-strongly convex, i.e., there exists $\alpha_z > 0$ such that
	\[
		L_t(z_1) \geq L_t(z_2) + \nabla L_t(z_2)^\top (z_1-z_2) + \frac{\alpha_z}{2} \norm{z_1-z_2}^2,
	\]
	\item $l_z$-smooth, i.e., there exists $l_z > 0$ such that
	\[
		L_t(z_1) \leq L_t(z_2) + \nabla L_t(z_2)^\top (z_1-z_2) + \frac{l_z}{2} \norm{z_1-z_2}^2,
	\]
	\item and Lipschitz continuous with Lipschitz constant $L_z$, i.e., there exists $L_z > 0$ such that
	\[
		\norm{L_t(z_1)-L_t(z_2)} \leq L_z \norm{z_1 - z_2},
	\]
\end{itemize}
for all $t \in \mathbb I_{\geq 0}$ and any two points $z_1,z_2 \in \mathbb R^{m+p}$.
\end{assumption}
\begin{remark} We assume Lipschitz continuity for clarity of exposition of our results, even though $l_z$-smoothness and Lipschitz continuity cannot be satisfied \textit{globally} simultaneously. However, if $u_t$ and $y_t$ remain within bounded sets for all time, Assumption~\ref{assump:cost_fcn} is satisfied on this bounded set. Moreover, techniques from~\cite{Nonhoff2020} can be used to avoid assuming Lipschitz continuity. In this case, all triangle inequalities in the proof of Theorem~\ref{thm:regret_bound} are replaced by Jensen's inequality which entails additional assumptions on the step size and the condition number $l_z/\alpha_z$ of the cost functions. Moreover, changing the regret definition below to $\mathcal R = \sum_{t=0}^T \norm{(u_t,y_t) - (\eta_t,\theta_t)}$ also removes the necessity to assume Lipschitz continuity of the cost functions. \hfill $\square$
\end{remark}
Characterizing the solution to \eqref{eq:OCP}, i.e., $u^*$ and $y^*$, for general time-varying cost functions $L_t$ requires optimization or verifying certain dissipativity conditions \cite{Muller2015,Grune2016,Grune2018} and is thus computationally expensive. For a priori unknown cost functions as considered in this work, computing $u^*$ and $y^*$ online is impossible altogether. Instead, we adopt a strategy of tracking the a priori unknown time-varying optimal states given by
\[
	(\eta_t, \theta_t) = \begin{cases} \arg \min_{u,y} & L_t(u,y) \\ \text{s.t. } &x = Ax + Bu \\ &y = Cx+Du \end{cases},
\]
where we define $\zeta_t = \begin{bmatrix} \eta_t^\top & \theta_t^\top \end{bmatrix}^\top$, $\eta_t\in\mathbb R^m$ is the optimal steady-state input, and $\theta_t\in\mathbb R^p$ is the optimal steady-state output of system~\eqref{eq:system_ss} at time~$t$. In case of constant convex cost functions $L$, steady-state operation is optimal \cite{Angeli2009}; hence, we expect that the proposed strategy yields good performance in many practical applications, in particular in case the cost functions $L_t$ do not change too frequently. Note that the setting considered here includes as a special case our previous works \cite{Nonhoff2020,Nonhoff2021,Nonhoff2021CDC}, where only strongly convex, smooth cost functions were considered that are each positive definite with respect to some (time-varying) steady state $(\eta_t,\theta_t)$ of the system. Here, we consider more general convex cost functions that do not need to satisfy this requirement. Such cost functions often occur in practice related to some economic considerations, such as minimization of energy cost (compare the example in Section~\ref{sec:numerical_example}), which is why such cost functions have been termed \textit{economic} in the context of model predictive control (see, e.g., \cite{Rawlings2017,Faulwasser2018,Rawlings2012}).

As common in OCO, we analyze our controller's closed-loop performance in terms of regret. In light of our strategy of tracking a priori unknown and time-varying optimal steady states of system~\eqref{eq:system_ss}, we define the regret $\mathcal R$ as
\[
	\mathcal R := \sum_{t=0}^T L_t(u_t,y_t) - L_t(\eta_t,\theta_t), \numberthis \label{eq:def_regret}
\]
i.e., the accumulated difference between the closed-loop cost of our controller and the optimal steady-state cost in hindsight. The regret~$\mathcal R$ is a measure of the performance lost due to not knowing the cost functions $L_t$ a priori. In the literature, commonly the goal is to achieve sublinear regret\footnote{In contrast to the classical OCO literature, we need to take the $\limsup$ instead of $\lim$ here due to the economic cost function (compare, e.g., \cite{Rawlings2012})}, i.e., 
\[
	\limsup_{T\rightarrow\infty} \mathcal R/T = \limsup_{T\rightarrow\infty} \frac{1}{T} \sum_{t=0}^T L_t(u_t,y_t) - L_t(\eta_t,\theta_t) \leq 0.
\]
Hence, if the proposed algorithm achieves sublinear regret, then the closed-loop cost is asymptotically on average no worse than the optimal steady-state cost. Such a performance result is typically also considered in the context of economic model predictive control (MPC), compare, e.g., \cite{Angeli2009,Faulwasser2018}.

Since we do not assume knowledge of the system matrices $(A,B,C,D)$, we assume that we have access to measurement data in the form of a prerecorded input-output sequence $\{u^d_k,y^d_k\}_{k=0}^{N-1}$ and an upper bound on the system order $n$. Note that we require the true system output as data instead of the (noisy) measured system output. Such data can be obtained in practice when, e.g., the prior data is recorded in a laboratory setting using more accurate measuring instruments than during online operation.
\begin{assumption} \label{assump:noisefree_offline_data}
	The output data $y^d = \{y^d_k\}_{k=0}^{N-1}$ is noise free.
\end{assumption}
Moreover, we assume that the data sequence is persistently exciting as defined in Definition~\ref{def:persistently_exciting}.
\begin{definition} \label{def:persistently_exciting} A signal $\{u_k\}_{k=0}^{N-1}$, $u_k \in \mathbb R^m$, is called persistently exciting of order $L$ if $\text{rank}(H_L(u))=mL$.
\end{definition}
This definition allows to characterize all possible system trajectories of \eqref{eq:system_ss} using only Hankel matrices of the data sequence. This result was first published in the context of behavioral system theory \cite{Willems2005} and can be formulated in the classical state space setting as follows.
\begin{theorem} \label{thm:fundamental_lemma} \cite[Theorem~3]{Berberich2020ECC} Suppose $\{u^d_k, y^d_k\}^{N-1}_{k=0}$ is a trajectory of system \eqref{eq:system_ss}, where $u^d$ is persistently exciting of order $L+n$ and let Assumption~\ref{assump:noisefree_offline_data} be satisfied. Then, $\{\bar u_k,\bar y_k\}^{L-1}_{k=0}$ is a trajectory of \eqref{eq:system_ss} if and only if there exists $\alpha \in \mathbb R^{N-L+1}$ such that
\[
	\begin{bmatrix} H_L(u^d) \\ H_L(y^d) \end{bmatrix} \alpha = \begin{bmatrix} \bar u \\ \bar y \end{bmatrix}.
\]
\end{theorem}
As discussed above, we aim to track a series of a priori unknown steady states without access to a model of the system. Therefore, a data-driven definition of steady states is given in Definition~\ref{def:steady_state}.
\begin{definition} \label{def:steady_state} An input-output pair $(u^s,y^s)$ is an equilibrium of \eqref{eq:system_ss}, if the sequence $\{u_k,y_k\}_{k=0}^{n}$ with $(u_k,y_k) = (u^s,y^s)$ for all $k \in \mathbb I_{[0,n]}$ is a trajectory of \eqref{eq:system_ss}.
\end{definition}
Definition~\ref{def:steady_state} states that an input-output pair $(u^s,y^s)$ is an equilibrium of system~\eqref{eq:system_ss} if and only if a sequence consisting of $(u^s,y^s)$ for at least $n+1$ consecutive time steps is a trajectory of the system. We make use of Definition~\ref{def:steady_state} and the prerecorded data sequence to characterize the steady-state manifold of system~\eqref{eq:system_ss} in Lemma~\ref{lem:steady_states}.
\begin{lemma} \label{lem:steady_states} Let Assumption~\ref{assump:noisefree_offline_data} be satisfied. Assume that the sequence $u^d$ is persistently exciting of order $2n+1$. Then, the input-output pair $(u^s,y^s)$ is an equilibrium of \eqref{eq:system_ss} if and only if
\[
	S z^s = \begin{bmatrix} S_u & S_y \end{bmatrix} \begin{bmatrix} u^s \\ y^s \end{bmatrix} = 0,
\]
where $S = \left( H_{n+1}H_{n+1}^\dagger - I_{(m+p)(n+1)}\right) \begin{bmatrix} \hat I_m & 0 \\ 0 & \hat I_p \end{bmatrix}$, $H_{n+1} = \begin{bmatrix} H_{n+1} (u^d) \\ H_{n+1} (y^d) \end{bmatrix}$, $\hat I_m = 1_{n+1} \otimes I_m$, and $\hat I_p = 1_{n+1} \otimes I_p$.
\end{lemma}
\begin{proof}
	By Definition~\ref{def:steady_state} and Theorem~\ref{thm:fundamental_lemma}, $(u^s,y^s)$ is a steady state of \eqref{eq:system_ss} if and only if there exists $\nu \in \mathbb R^{N-n}$ such that
	\begin{equation}
		H_{n+1} \nu = \begin{bmatrix} H_{n+1} (u^d) \\ H_{n+1} (y^d) \end{bmatrix} \nu = \begin{bmatrix} \hat I_m u^s \\ \hat I_p y^s \end{bmatrix}. \label{eq:proof_lemma_1}
	\end{equation}
	The general solution to this equation is given by
	\[
		\nu = H_{n+1}^\dagger \begin{bmatrix} \hat I_m u^s \\ \hat I_p y^s \end{bmatrix} + (I_{N-n} + H_{n+1}^\dagger H_{n+1}) \nu',
	\]
	where $\nu' \in \mathbb R^{N-n}$ can be chosen arbitrary. Since the second term on the right-hand side of the above expression is in the nullspace of $H_{n+1}$, inserting $\nu$ back into~\eqref{eq:proof_lemma_1} yields
	\[
		\left( H_{n+1} H_{n+1}^\dagger - I_{(m+p)(n+1)} \right) \begin{bmatrix} \hat I_m &0 \\ 0 & \hat I_p \end{bmatrix} z^s = 0_{(m+p)(n+1)},
	\]
	which proves the result.
\end{proof}
Lemma~\ref{lem:steady_states} explicitly defines the steady-state manifold of system~\eqref{eq:system_ss} only in terms of Hankel matrices of the prerecorded data sequence. 

\section{ALGORITHM} \label{sec:algorithm}
In this section, we introduce our algorithm. For notational convenience, we define $U = H_{2n+\mu+1}(u^d)$ and \mbox{$Y = H_{2n+\mu+1}(y^d)$}, i.e., the Hankel matrices associated with the system input and output, respectively. In addition, we denote
\[
	H_\alpha = \begin{bmatrix} U^{1:n} \\ U^{n+1:2n+\mu+1} \\ Y^{1:n} \end{bmatrix}\text{, and } 
	H_\beta = \begin{bmatrix} U^{1:n} \\ U^{n+\mu+1:2n+\mu+1} \\ Y^{1:n} \\ Y^{n+\mu+1:2n+\mu} \end{bmatrix}.
\]
The proposed data-driven OCO scheme is given in Algorithm~1. In the framework described above, at every time instance $t$, Algorithm~1 
\begin{enumerate}
	\item computes an input $u_t$ via \eqref{eq:defhate}-\eqref{eq:OutputAlgo} and applies it to system~\eqref{eq:system_ss},
	\item measures the output $\tilde y_{t}$ and receives the cost function $L_{t}$,
	\item moves to time step $t+1$.
\end{enumerate} 

\begin{figure} 
	\vspace{0pt}
	{
		\fbox{\parbox{.94\linewidth}{
				\underline{Algorithm 1: Data-Driven Output Feedback}
				
				\vspace{7pt}
				
				\noindent Given step size $\gamma$, prediction horizon $\mu$, initialization $\hat u_{-1} \in \mathbb{R}^{m(\mu+1)}$, $z^s_{-1} \in \mathbb R^{m+p}$, $u_{[-n:-1]}$, $\tilde y_{[-n:-1]} - \hat e_{[-n:-1]}$, and data $(u^d,y^d)$. At each time $t$:
				\begin{align}
					&\text{If $t=0$, go to \eqref{eq:defalpha_algo}} \nonumber \\
					& \hat e_{t-1} = \tilde y_{t-1} - Y^{n+1} (\alpha_{t-1} + \beta_{t-1})  \label{eq:defhate}\\
					&\text{Choose $\alpha_t$  such that} \nonumber \\[-4pt]
					&\quad H_\alpha \alpha_t = \begin{bmatrix} u_{[t-n:t-1]} \\ \hdashline \sigma \hat u_{t-1} \\ 1_{n+1} \otimes u^s_{t-1} \\ \hdashline \tilde y_{[t-n:t-1]} - \hat e_{[t-n:t-1]} \end{bmatrix} \label{eq:defalpha_algo} \\
					&\hat z_{t}^\mu = \begin{bmatrix} u^s_{t-1} \\ Y^{n+\mu+1} \alpha_t \end{bmatrix} \label{eq:mu_ahead_prediction} \\
					&z^s_t = \begin{bmatrix} u^s_t \\ y^s_t \end{bmatrix} = \Big(I_{m+p} - S^\dagger S\Big) \Big( \hat z_{t}^\mu - \gamma \nabla L_{t-1}\left(\hat z_{t}^\mu\right) \hspace{-2pt} \Big) \label{eq:OGD}\\
					&\beta_t \hspace{-3pt} = \hspace{-3pt}
					\begin{cases}
						&\hspace{-10pt} \arg\min_\beta \hspace{5ex}{\norm{Q\beta}} \\
						&\hspace{-10pt}\text{s.t. } H_\beta \beta = \hspace{-2pt} \begin{bmatrix} 0_{mn} \\  1_{n+1} {\otimes} u_t^s {-} U^{n+\mu+1:2n+\mu+1} \alpha_t  \\ 0_{pn} \\  1_{n} {\otimes} y_t^s - Y^{n+\mu+1:2n+\mu}\alpha_t  \end{bmatrix}
					\end{cases} \label{eq:defbeta_algo} 
					\\
					&\hat u_t = U^{n+1:n+\mu+1}(\alpha_t + \beta_t) \label{eq:PredInputs} \\
					&u_t =U^{n+1} (\alpha_t + \beta_t) \label{eq:OutputAlgo}
				\end{align}
				Measure $\tilde y_{t}$ and receive $L_{t}$ \\
				Set $t = t+1$ and go to \eqref{eq:defhate}
		}}
	}
	\vspace{-10pt}
\end{figure}

Roughly speaking, Algorithm~1 estimates the measurement noise by relying on its own predictions, applies online gradient descent (OGD) to estimate the optimal equilibrium of system~\eqref{eq:system_ss}, and calculates an input sequence that reaches the estimated optimal steady state. The whole procedure is illustrated in Fig.~\ref{fig:algo_sketch}.

In more detail, an estimate of the measurement noise is computed in~\eqref{eq:defhate} by comparing the measured output $\tilde y_{t-1}$ to the output predicted at the previous time step $Y^{n+1}(\alpha_{t-1} + \beta_{t-1})$. The estimated measurement noise is then used in combination with the last $n$ inputs $u_{[t-n:t-1]}$ and outputs $\tilde y_{[t-n:t-1]}$ in~\eqref{eq:defalpha_algo} to initialize a prediction step. As an input for prediction, we take the shifted previously predicted input sequence $\sigma \hat u_{t-1}$ and append it with the previously estimated optimal steady-state input $u^s_{t-1}$. Thus, $\alpha_t$ in~\eqref{eq:defalpha_algo} encodes the prediction at time $t$. In~\eqref{eq:mu_ahead_prediction}, the previously estimated steady-state input $u^s_{t-1}$ and the $\mu$-step ahead prediction $Y^{n+\mu+1}\alpha_t$ are collected in preparation for the projected OGD step in~\eqref{eq:OGD}, where the parameter $\mu$ can be interpreted as the prediction horizon of Algorithm~1. As common in OGD, we perform one gradient descent step in~\eqref{eq:OGD} based on the previous cost function $L_{t-1}$, since we do not have access to the current cost function $L_t$ yet. Note that multiplication by $I_{m+p} - S^\dagger S$ is equivalent to orthogonal projection onto the null space of $S$, which corresponds to the steady-state manifold by Lemma~\ref{lem:steady_states}. Thus, in~\eqref{eq:OGD}, we perform one online gradient descent step and project it onto the steady-state manifold of system~\eqref{eq:system_ss}. The resulting input-output pair $z^s_t$ can be regarded as an estimate for the optimal steady state. In~\eqref{eq:defbeta_algo}, we compute an input sequence which, if applied in addition to the input sequence used for prediction (i.e., $U\alpha_t$), reaches the estimated optimal steady state $z^s_t$ in $\mu$ steps and remains at $z_t^s$ for another $n+1$ steps in order to ensure that the unknown internal states $x_t$ of system~\eqref{eq:system_ss} reach the desired steady state. In order to be able to reach the estimated optimal steady state $z^s_t$ in $\mu$ time steps, we require the prediction horizon $\mu$ to be sufficiently long.
\begin{assumption} \label{assump:pred_horizon} The prediction horizon satisfies $\mu \geq \mu^*$, where $\mu^*$ is the controllability index of system~\eqref{eq:system_ss}, i.e.,
\[
	\text{rank}(\begin{bmatrix} B & AB & \dots & A^{\mu^*-1}B \end{bmatrix}) = n.	
\]
\end{assumption}
Note that $n \geq \mu^*$ always holds. Since we require an upper bound of the system order to be available, it is therefore possible to satisfy Assumption~\ref{assump:pred_horizon} without knowing $\mu^*$. However, simulations suggest that a shorter prediction horizon can sometimes be beneficial for the algorithm's performance, since decreasing the prediction horizon forces the algorithm to reach the desired steady state $z^s$ in less steps in~\eqref{eq:defbeta_algo}, resulting in a more aggressive controller. Finally, we update the predicted input sequence $\hat u_t$ in~\eqref{eq:PredInputs}. Note that $U^{n+1:n+\mu+1} \alpha_t = \begin{bmatrix} \sigma\hat u_{t-1} \\ u^s_{t-1} \end{bmatrix}$. Thus, the predicted input sequence is updated by shifting it, appending $u^s_{t-1}$, and adding the input sequence encoded in $\beta_t$, which steers the system to the new estimate of the optimal steady state $z_t^s$. Last, the first part of $\hat u_t$ is applied in~\eqref{eq:OutputAlgo} to system~\eqref{eq:system_ss}. Then, we measure the new (noisy) system output $\tilde y_t$, receive the cost function~$L_t$, and move to the next time step $t+1$.

The matrix $Q$ in the cost function of~\eqref{eq:defbeta_algo} can be tuned to achieve satisfactory performance, e.g., $Q=I$ minimizes the norm of $\beta$ and can be beneficial if there is process noise affecting the system~\eqref{eq:system_ss}, $Q=U$ minimizes the input difference needed to steer the system to $z^s_t$ (instead of $z^s_{t-1}$), and $Q=Y$ similarly minimizes the deviation of the system's output from the predicted output. Moreover, weighted combinations are possible by stacking the matrices $I$, $U$, and $Y$ in $Q$ (compare Section~\ref{sec:numerical_example}).

\begin{figure}
	\centering \small
	\begin{tikzpicture}
	\definecolor{mygreen}{rgb}{0.2,0.7,0.2}
	\clip (0,0) rectangle (8,5);
	
	\draw[red,dotted,thick] (0,2.5) -- (5,0); 
	
	\coordinate (theta) at (1,2);
	\draw (0.95,1.95) -- (1.05,2.05)
		  (0.95,2.05) -- (1.05,1.95);
	\node[fill=white,inner sep = 0] at (1.2,2.2) {$\theta_{t-1}$};
	\draw[dotted,rotate around={115:(theta)}] (theta) ellipse (2 and 1);
	\draw[dotted,rotate around={115:(theta)}] (theta) ellipse (4 and 2);
	\draw[dotted,rotate around={115:(theta)}] (theta) ellipse (6 and 3);
	\draw[dotted,rotate around={115:(theta)}] (theta) ellipse (8 and 4);
	\draw[dotted,rotate around={115:(theta)}] (theta) ellipse (10 and 5);
	
	\coordinate (yt-1) at (7,4);
	\node[fill=white,inner sep = 0] at (7.2,4.2) {$y_{t-1}$};
	\draw (6.95,3.95) -- (7.05,4.05)
		  (6.95,4.05) -- (7.05,3.95);
	
	\coordinate (haty) at (3,1);
	\node[fill=white,inner sep = 0] at (3.9,1.2) {$Y^{n+\mu+1}\alpha_t$};
	\draw (2.95,0.95) -- (3.05,1.05)
		  (2.95,1.05) -- (3.05,0.95);
	\node[fill=white,inner sep = 0] at (2.5,0.7) {$-\gamma\nabla L_{t-1}$};
		  
	\draw (yt-1) -- (6,4.25) -- (4.25,3) -- (3.2,1.5) -- (haty);
	\node[fill=white,inner sep = 0] at (4.4,2.8) {(a)};
	
	\coordinate (grad) at (2,0.9);
	\coordinate (ys) at (2.24,1.38);
	\draw[thick,->,blue,dashed] (haty) -- (grad);
	\draw[thick, dashed,blue] (grad) -- (ys);
	\draw (2.19,1.33) -- (2.29,1.43)
		  (2.19,1.43) -- (2.29,1.33);
	\node[fill=white,inner sep = 0] at (2.1,1.6) {$y^s_t$};
	
	\coordinate (yt) at (6,4.3);
	\draw (5.95,4.25) -- (6.05,4.35)
		  (5.95,4.35) -- (6.05,4.25);
	\node[fill=white,inner sep = 0] at (6,4.5) {$y_t$};
	\draw[thick,mygreen] (yt-1) -- (yt);
	\draw (yt) -- (4.35,3.5) -- (2.5,2) -- (ys);
	\node[fill=white,inner sep = 0] at (3.1,2.8) {(b)};
\end{tikzpicture}
	\caption{Schematic illustration of Algorithm~1. The previous cost function is depicted by its level sets (dotted) together with the steady-state manifold (red, dotted). Algorithm~1 predicts the output $\mu$-steps ahead (a), performs one projected online gradient descent step (blue, dashed), updates the input sequence (b), and applies the first part of the updated sequence to the system (green).}
	\label{fig:algo_sketch}
	\vspace{-15pt}
\end{figure}
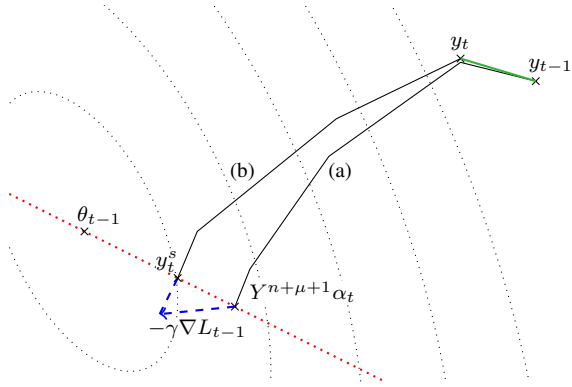 

Since we compute input-output sequences of length \mbox{$2n+\mu+1$} in Algorithm~1 ($n$ steps for initialization, $\mu$ steps for prediction, and $n+1$ steps as a terminal constraint ensuring steady-state operation), by Theorem~\ref{thm:fundamental_lemma} we need persistency of excitation of order $3n+\mu+1$.
\begin{assumption} \label{assump:persistency_excitation}
	The input $u^d$ of the data sequence is persistently exciting of order $3n+\mu+1$.
\end{assumption}
Note that persistency of excitation of order $3n+\mu+1$ requires a data sequence of length 
\[
	N \geq (m+1)(3n+\mu+1)-1.
\]

Finally, we derive explicit formulas to solve \eqref{eq:defalpha_algo} and \eqref{eq:defbeta_algo} in Algorithm~1. In order to do so, we need to ensure that \eqref{eq:defalpha_algo} and \eqref{eq:defbeta_algo} always have a feasible solution, which is guaranteed if their respective right-hand sides describe (parts of) valid input-output sequences of system~\eqref{eq:system_ss} by Theorem~\ref{thm:fundamental_lemma}.
\begin{lemma} \label{lem:feasibility} Let Assumptions~\ref{assump:system_properties}, \ref{assump:noisefree_offline_data}-\ref{assump:persistency_excitation} be satisfied and assume that the initialization $u_{[-n:-1]}$, $\tilde y_{[-n,-1]} - \hat e_{[-n,-1]}$ is an input/output sequence of system~\eqref{eq:system_ss}. Then, \eqref{eq:defalpha_algo} and \eqref{eq:defbeta_algo} have a feasible solution for all $t \in \mathbb I_{\geq 0}$.
\end{lemma}
\begin{proof}
Assume that at time step $t$, $u_{[t-n:t-1]}$ and $\tilde y_{[t-n:t-1]} - \hat e_{[t-n:t-1]}$ are a valid $n$-step trajectory of system~\eqref{eq:system_ss}. Then, $\alpha_t$ can be chosen according to~\eqref{eq:defalpha_algo}, compare~\cite{Markovsky2008}. Moreover, there exists $\beta_t'$ such that 
\[
	H_\beta \beta_t' = \begin{bmatrix} 0_{mn} \\ 1_{n+1} \otimes u^s_t \\ 0_{pn} \\ 1_n \otimes y^s_t \end{bmatrix}
\]
by Assumption~\ref{assump:pred_horizon}, since system~\eqref{eq:system_ss} can be steered from $0$ to any steady state in $\mu$ steps by controllability, and $\beta_t''$ such that 
\[
	H_\beta \beta_t'' = \begin{bmatrix} 0_{mn} \\ U^{n+\mu+1:2n+\mu+1} \alpha_t \\ 0_{pn} \\ Y^{n+\mu+1:2n+\mu} \alpha_t \end{bmatrix}
\]
by Assumption~\ref{assump:pred_horizon} and because $\alpha_t$ encodes a valid input-output sequence. Since sums of input-output sequences of a linear time-invariant system are input-output sequences of the same system, there exists a solution to~\eqref{eq:defbeta_algo} at time step $t$ given by $\beta_t = \beta'_t - \beta_t''$. Then, at time step $t+1$, the right-hand side of \eqref{eq:defalpha_algo} is a valid input-output sequence because of Theorem~\ref{thm:fundamental_lemma} and
\begin{align*}
	U^{2:n+1} (\alpha_{t} + \beta_{t}) &\refeq{\eqref{eq:defalpha_algo},\eqref{eq:OutputAlgo}} u_{[t-n+1:t]} \\
	Y^{2:n+1} (\alpha_{t} + \beta_{t}) &\refeq{\eqref{eq:defhate}} \tilde y_{[t-n+1:t]} - \hat e_{[t-n+1:t]}.
\end{align*}
Thus, \eqref{eq:defalpha_algo} and \eqref{eq:defbeta_algo} always have a solution by induction if the algorithm is initialized with a valid input-output sequence $u_{[-n:-1]}$ and $\tilde y_{[-n,-1]} - \hat e_{[-n,-1]}$. 
\end{proof}
Lemma~\ref{lem:feasibility} states that we need a feasible initialization $u_{[-n:-1]}$, $\tilde y_{[-n:-1]}-\hat e_{[-n:-1]}$. Thus, in the following, we assume that Algorithm~1 is initialized correctly. This can be ensured, e.g., by choosing $u_{[-n:-1]} = 0$ and \mbox{$\hat e_{[-n:-1]} = \tilde y_{[-n:-1]}$} or by solving
\begin{align*}
	(\alpha_0,\hat e_{[-n:-1]}) &= \\
	\arg \min_{\alpha,\hat e} \quad &\norm{Y^{1:n} \alpha - (\tilde y_{[-n:-1]} - \hat e_{[-n:-1]})} {+} \lambda \norm{\begin{bmatrix} \alpha \\ \hat e\end{bmatrix}} \\
	\text{s.t.} \quad &U\alpha = \begin{bmatrix} u_{[-n:-1]} \\ \sigma \hat u_{-1} \\ 1_{n+1} \otimes u^s_{t-1} \end{bmatrix},
\end{align*}
where $\lambda \in \mathbb R_{\geq 0}$ is a weighting factor, for some initialization $\hat u_{-1}$, $u^s_{-1}$, instead of solving \eqref{eq:defhate}~-~\eqref{eq:defalpha_algo} at time step $t=0$.

Note that one solution to \eqref{eq:defalpha_algo} is given by the pseudo-inverse
\[
	\alpha_t = H_\alpha^\dagger \begin{bmatrix} u_{[t-n:t-1]} \\ \hdashline \sigma \hat u_{t-1} \\ 1_{n+1} \otimes u^s_{t-1} \\ \hdashline \tilde y_{[t-n:t-1]} - \hat e_{[t-n:t-1]} \end{bmatrix}. \numberthis \label{eq:defalpha}
\]
Moreover, if $Q^\top Q \succ 0$, i.e., $Q^\top Q$ is positive definite, then the unique solution to \eqref{eq:defbeta_algo} is given by the weighted pseudo-inverse \cite{Elden1982}
\[
	\beta_t = \left( I_{N-2n-\mu} - \left( Q \left( I_{N-2n-\mu} - H_\beta^\dagger H_\beta \right)\right)^\dagger Q \right) H_\beta^\dagger g_t, \numberthis \label{eq:defbeta} 
\]
where $g_t$ is the right-hand side of \eqref{eq:defbeta_algo}
\[
	g_t = \begin{bmatrix} 0_{mn} \\  1_{n+1} {\otimes} u_t^s {-} U^{n+\mu+1:2n+\mu+1} \alpha_t  \\ 0_{pn} \\  1_{n} {\otimes} y_t^s - Y^{n+\mu+1:2n+\mu}\alpha_t  \end{bmatrix}. \numberthis \label{eq:defgt}
\]
If $Q^\top Q \succeq 0$ is only positive semidefinite, then the solution to~\eqref{eq:defbeta_algo} is not unique and~\eqref{eq:defbeta} is only one possible solution. In the following, we assume that~\eqref{eq:defbeta_algo} is solved using~\eqref{eq:defbeta} in both cases. Thus, the necessary online computations in Algorithm~1 reduce to one gradient evaluation and multiple matrix-vector multiplications.

\section{THEORETICAL RESULTS} \label{sec:theoretical_results}
In this section, we discuss theoretical guarantees for Algorithm~1, in particular a bound on the regret~$\mathcal R$. In order to derive such a bound, we first analyze the error estimates $\hat e$. Lemma~\ref{lem:error_estimate_convergence} states that the measurement error estimates $\hat e$ converge to the true measurement error $e$. Thus, Algorithm~1 is able to (asymptotically) exactly recover the measurement error $e$ and control the true system output $y_t$, even though only noisy measurements $\tilde y_t$ are available at each time step.

\begin{lemma} \label{lem:error_estimate_convergence} Let Assumptions~\ref{assump:system_properties},~\ref{assump:noisefree_offline_data}-\ref{assump:persistency_excitation} be satisfied and assume that the initialization $u_{[-n:-1]}$, $\tilde y_{[-n,-1]} - \hat e_{[-n,-1]}$ is an input/output sequence of system~\eqref{eq:system_ss}. Then, the error of the measurement noise estimates $\hat e - e$ follows the unforced system dynamics, i.e., $\hat e - e$ is an output of \eqref{eq:system_ss} with $u \equiv 0$ and
\[
	\lim_{t\rightarrow\infty} \hat e_t - e_t = 0.
\]
\end{lemma}
\begin{proof}
	For every $t\geq 0$, let
	\[
		\alpha_t^* = H_\alpha^\dagger\begin{bmatrix} u_{[t-n:t-1]} \\ \sigma \hat u_{t-1} \\ 1_{n+1} \otimes u^s_{t-1} \\ y_{[t-n:t-1]} \end{bmatrix},
	\]
	i.e., the prediction with the real outputs $y_{[t-n:t-1]}$, compare~\eqref{eq:defalpha_algo}, and
	\[
		\epsilon_t = \alpha_t - \alpha_t^*. \numberthis \label{eq:def_epsilon}
	\]
	Note that $\alpha_t$ is well-defined at all times due to Lemma~\ref{lem:feasibility}. Then, we have $U\epsilon_t = 0$ by definition of $\epsilon_t$ and
	\begin{align*}
	Y^{1:n} \epsilon_t &\refeq{\eqref{eq:defalpha_algo}} \tilde y_{[t-n:t-1]} - \hat e_{[t-n:t-1]} - y_{[t-n:t-1]} \\
	&\refeq{\eqref{eq:system_ss}} e_{[t-n:t-1]} - \hat e_{[t-n:t-1]}. \numberthis \label{eq:estimate_recursion}
	\end{align*}
	Moreover, $y_{t-1} = Y^{n+1} (\alpha_{t-1}^* + \beta_{t-1})$ by Theorem~\ref{thm:fundamental_lemma}, i.e., the coefficients $\alpha_{t-1}^* + \beta_{t-1}$ encode the true initialization $(u_{[t-n-1:t-2]},y_{[t-n-1:t-2]})$ and input $U^{n+1} (\alpha_{t-1}^* + \beta_{t-1}) = U^{n+1} (\alpha_{t-1} + \beta_{t-1}) \refeq{\eqref{eq:OutputAlgo}} u_{t-1}$ and, therefore, predict the output $y_{t-1}$ correctly. Combining this fact with~\eqref{eq:estimate_recursion} yields
	\begin{align*}
		Y^{n} \epsilon_t &= \tilde y_{t-1} - \hat e_{t-1} - y_{t-1} \\
		&\refeq{\eqref{eq:defhate}} Y^{n+1} (\alpha_{t-1}+\beta_{t-1}) - y_{t-1} \\
		&= Y^{n+1} (\alpha_{t-1} - \alpha_{t-1}^*) = Y^{n+1} \epsilon_{t-1},
	\end{align*}
	which implies
	\begin{equation}
		Y^{1:n} \epsilon_t \refeq{\eqref{eq:estimate_recursion}} Y^{2:n+1} \epsilon_{t-1}. \label{eq:error_error_dynamics}
	\end{equation}
	Combining the above results $U\epsilon_t = 0$ and \eqref{eq:error_error_dynamics}, we conclude that the error sequence $e_t - \hat e_{t}$ follows the unforced system dynamics. In more detail, at each time step $t$ the sequence generated by $Y\epsilon_t$ by~\eqref{eq:error_error_dynamics} is initialized by the endpiece of the initialization of $Y\epsilon_{t-1}$ (i.e., $Y^{2:n}\epsilon_{t-1}$) appended with the one step ahead prediction at time $t-1$ (i.e., $Y^{n+1}\epsilon_{t-1}$). Therefore, we have by Theorem~\ref{thm:fundamental_lemma} and $U\epsilon_t=0$ for all $t$ that $e_t - \hat e_t$ is the output of a trajectory of the unforced system for all $t$. Since the unforced system dynamics are stable by Assumption~\ref{assump:system_properties}, we obtain the result $\lim_{t\rightarrow\infty} e_t - \hat e_t = 0$.
\end{proof}

Next, we are able to derive an upper bound on the regret~$\mathcal R$ as stated in Theorem~\ref{thm:regret_bound}.

\begin{theorem} \label{thm:regret_bound} Let Assumptions~\ref{assump:system_properties}-\ref{assump:persistency_excitation} be satisfied and choose $0 < \gamma \leq \frac{2}{\alpha_z+l_z}$. Moreover, assume that the initialization $u_{[-n:-1]}$, $\tilde y_{[-n,-1]} - \hat e_{[-n,-1]}$ is an input/output sequence of system~\eqref{eq:system_ss}. Then, the regret~\eqref{eq:def_regret} can be upper bounded by
\[
	\mathcal R \leq C_\mu + C_\zeta \sum_{t=0}^T \norm{\zeta_{t}-\zeta_{t-1}} + C_e E_0,	
\]
where $E_0 = \norm{e_{[-n:-1]} - \hat e_{[-n:-1]}}$ and $C_\mu, C_\zeta, C_e < \infty$ are constants independent of $T$ and $\zeta_{-1} = z^s_{-1}$.
\end{theorem}
The proof is given in the appendix. The upper bound on the regret depends on constants, which in turn depend on system and problem parameters, $\sum_{t=0}^T \norm{\zeta_{t}-\zeta_{t-1}}$, and $E_0 = \norm{e_{[-n:-1]} - \hat e_{[-n:-1]}}$, i.e., the initialization error of the measurement error estimates. The quantity $\sum_{t=0}^T \norm{\zeta_{t}-\zeta_{t-1}}$, commonly termed path length in the literature \cite{Li2021a}, can be regarded as a measure of the variation of the cost functions. A bound which depends on the variation of the cost functions is to be expected, since in our framework, the cost function $L_t$ is only available to the algorithm at time step $t+1$, i.e., there is a one-step delay between the cost function becoming active and being used to control the system. Thus, it is impossible to achieve low regret if the cost functions vary too frequently. A bound which depends linearly on $\sum_{t=0}^T \norm{\zeta_{t}-\zeta_{t-1}}$ is well aligned with other results on dynamic regret in the literature, compare, e.g., \cite{Nonhoff2021,Li2019}. A sublinear regret can therefore be achieved if the path length is sublinear in $T$. Moreover, introduction of measurement noise to the control problem only introduces an additional \textit{constant} term $C_e E_0$ in the regret upper bound compared to \cite{Nonhoff2021CDC}. This is due to the convergence of $\hat e$ to $e$ as shown in Lemma~\ref{lem:error_estimate_convergence}. 

Finally, consider the case where the optimal steady state is constant, i.e., $\zeta_t = \zeta_{t'}$ for some $t' \in \mathbb I_{\geq 0}$ and all $t\geq t'$. Following the proof of Theorem~\ref{thm:regret_bound}, it can be shown that
\[
	\sum_{t=0}^T \norm{z_t - \zeta_t} \leq C_\mu' + C_\zeta' \sum_{t=0}^T \norm{\zeta_{t}-\zeta_{t-1}} + C_e E_0,
\]
where $C_\mu', C_\zeta' < \infty$ are again constants independent of $T$. Thus,
\[
	\sum_{t=t'+1}^T \norm{z_t - \zeta_t} \leq C_\mu' + C_e E_0,
\]
which implies that, in the case that the optimal steady state is constant, the closed loop with Algorithm~1 asymptotically converges to the optimal steady state.

\begin{remark} \label{rem:unstable_systems} (Unstable systems)
	Suppose that Assumption~\ref{assump:system_properties} is not satisfied because the system is not Schur stable. In this case, it is possible to design a linear stabilizing feedback \cite{DePersis2020,Berberich2020arXiv} and apply Algorithm~1 to the stabilized system, as discussed above. In particular, at each time step $t$, $u_t + v_t$ is applied to the system, where $u_t$ is the input computed in \eqref{eq:OutputAlgo} and 
	\begin{align*}
		v_t &= v^y_t + v^e_t = K\begin{bmatrix} v_{[t-n:t-1]} \\ \tilde y_{[t-n:t-1]} \end{bmatrix} \\ &= K\begin{bmatrix} v_{[t-n:t-1]}^y \\ y_{[t-n:t-1]} \end{bmatrix} + K\begin{bmatrix} v_{[t-n:t-1]}^e \\ e_{[t-n:t-1]} \end{bmatrix}, 
	\end{align*}
	where $K$ is the stabilizing controller. In order to do so, every Hankel matrix of the open-loop system in Algorithm~1 has to be replaced by Hankel matrices of the stabilized system. Moreover, a mapping $V(y^s)$ can be computed that maps a steady-state output $y^s$ of the system to a steady-state input of the stabilizing controller. Finally, the cost functions have to be reformulated to $L_t(u+V(y),y)$ to account for the stabilizing input when determining the optimal steady state. Then, it is still possible to derive a regret upper bound for the prestabilized system, but the theoretical guarantees deteriorate in two ways:
	\begin{enumerate}
		\item The estimates of the measurement error do not asymptotically exactly converge to the true measurement error as stated in Lemma~\ref{lem:error_estimate_convergence}, because the stabilized system is only practically (and not asymptotically) stable due to the measurement noise. Instead, the estimates $\hat e$ inherit their stability properties from the stabilized system. For example, assume that the stabilizing feedback stabilizes some robust positive invariant (RPI) set. In this case, the estimates $\hat e$ converge to the same RPI set around the true measurement error $e$.
		\item The regret bound is increased by additional terms $\norm{Y^{n+\mu+1} H_\alpha^\dagger} \left( C_{sl} T v_{max}+C_{sc} v_{max} \right)$, where the constants $C_{sl}$ and $C_{sc}$ only depend on system parameters $A,B,C,D,n$, and the prediction horizon $\mu$, and $v_{max}$ is an upper bound on the error feedback $\norm{v^e_t} \leq v_{max}$ for all $t$. Thus, the regret upper bound for the prestabilized system becomes \textit{linear} in $T$, which is to be expected since the stabilizing controller feeds back the measurement error $e_t$ at every time step, thereby preventing us from staying at the optimal steady state. In this regard, the feedback of the measurement error $v^e_t$ can also be interpreted as process noise acting on a stable system. \hfill $\square$
	\end{enumerate}
\end{remark}

\section{APPLICATION EXAMPLE - THERMAL CONTROL} \label{sec:numerical_example}
\subsection{Setting}
In this section, we test our OCO-based control scheme on a thermal control problem. Specifically, we consider a Heating Ventilation and Air Conditioning (HVAC) system which controls the temperature of five nonuniform zones. The HVAC system is equipped with a sensor in zone~1,~4, and~5, and actuators adjusting the supply air rate in every zone. The zones are depicted in Fig.~\ref{fig:zone_placement}. 
\begin{figure}
	\centering \small
	\begin{tikzpicture}

\draw (0,0) rectangle (4,4);
\clip (0,0) rectangle (4,4);
\draw (0,2) -- (3.3,2) -- (3.3,3)
	  (2.4,0) -- (2.4,4)
	  (2.4,3) -- (4,3);	
\node[circle,draw=black] (z1) at (1.25,3) {Zone 1};	  
\node (z2) at (3.25,3.5) {Zone 2};
\node (z3) at (2.875,2.5) {Zone 3};
\node[circle,draw=black] (z4) at (1.25,1) {Zone 4};
\node[circle,draw=black] (z5) at (3.25,1) {Zone 5};

\end{tikzpicture}
	\caption{Schematic illustration of the 5 zones controlled by an HVAC system. Measured zones are indicated by circles.}
	\label{fig:zone_placement}
	\vspace{0pt}
\end{figure}
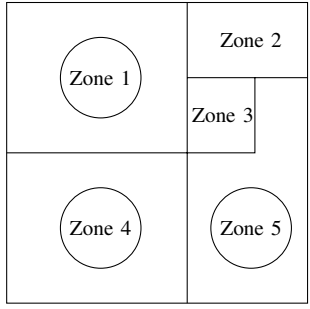
We consider the linear thermal dynamics model proposed in \cite{Li2021,Zhang2016} given by
\[
	C_i \dot T_i = \frac{T^o - T_i}{R_i} + \sum_{j\in\mathcal N(i)} \frac{T_j-T_i}{R_{ij}} + u_{i,t} + q_{i,t},
\]
where $C_i$ is the thermal capacitance of zone~$i$, $T_i$ is the zone temperature of zone~$i$, $T^o$ is the outdoor temperature, $R_i$ is the thermal resistance between the $i$-th zone and outside, $R_{ij}$ is the thermal resistance between zones $i$ and $j$, $\mathcal N(i)$ denotes the set of zones neighboring zone~$i$, $u_{i,t}$ is the control input at time $t$ associated with zone $i$, and $q_{i,t}$ denotes (unknown) process noise, caused, e.g., by additional heat sources in zone $i$. For zone~3, we set $R_3 = \infty$ in our simulation since it is surrounded by other zones and, therefore, not directly influenced by the outdoor temperature. Note that we do not consider process noise in our theoretical work, but do consider it in the simulation as an additional difficulty. We define the system states as $x_t = \begin{bmatrix} x_{1,t} & x_{2,t} & x_{3,t} & x_{4,t} & x_{5,t} \end{bmatrix}^\top \in \mathbb R^5$, where $x_{i,t} = \Delta T_i = T_i - T^o$ denotes the difference between the temperature of the $i$-th zone and the outside temperature at time $t$. Since there are sensors only in zones~1,~4, and~5, but an actuator in every zone, we set
\[
	C = \begin{bmatrix} 1 & 0 & 0 & 0 & 0 \\ 0 & 0 & 0 & 1 & 0 \\ 0 & 0 & 0 & 0 & 1 \end{bmatrix}, \quad B = I_{n=5}, \quad D=0.
\]
Then, we discretize the thermal dynamics with sample time $t_s = 60$\,s. The cost function consists of a term penalizing the deviation from a desired temperature $T^{set}_t$ and a term minimizing control cost
\[
	L_t(u,y) = \frac{1}{2} (y-\Delta T^{set}_t)^\top\Lambda_t(y-\Delta T^{set}_t) + \frac{\lambda_t p_t}{2} \norm{u}^2,
\]
where $\Delta T^{set}_t = T^{set}_t - T^o$, $\Lambda_t\in\mathbb R^{3\times3}$ $\lambda_t \in \mathbb R$ are a priori unknown time-varying parameters, and $p_t$ denotes the a priori unknown energy cost. In particular, $\lambda_t$ and $\Lambda_t$ are weighting factors, trading off user comfort and control cost. We set $T^o = 15$\,°C, $\lambda_t = 10$, $\Lambda_t = I_p$. However, we change $\Lambda_t$ to $\Lambda_t = 0.1 I_p$ between $0$\,am and $6$\,am, in order to save energy during the night. The normalized energy cost $p_t$ is shown in Figure~\ref{fig:energy_cost}. We choose $T^{set}_t = 18\text{\,°C} \cdot 1_p$ but switch it, a priori unbeknown to the algorithm, at $9$\,am to \mbox{$T^{set}_{t} = 21\text{\,°C}\cdot 1_p$.} In Algorithm~1, we choose \mbox{$Q = \begin{bmatrix} I_{N-2n-\mu} & U^\top \end{bmatrix}^\top$,} \mbox{$\mu\in\{10,30\}$,} and $\gamma = 0.15$, which satisfies $\gamma \leq \frac{2}{l_z+\alpha_z}$. At $t=n=5$, we initialize the algorithm with $\hat u_{4}=0_{\mu+1}$, $z^s_{4} = 0_{m+p}$, $u_{[0:4]}=0_{mn}$, and $\hat e_{[0:4]} = \tilde y_{[0:4]}$, while the real (unknown) initial condition is $x_0 + T^o = 17$\,°C for each zone. Note that Algorithm~1 does not control the system for the first $n=5$ time steps. Finally, we sample $q_{t,i}$ uniformly from the interval $[-0.1,0.1]$.

\subsection{Prediction Horizon and Robustness to Measurement Noise}

First, we simulate the proposed Algorithm~1 with different prediction horizons and assess its robustness with respect to measurement noise. To this end, we sample the measurement error $e_t$ uniformly from the interval $[-1,1]$. Moreover, we increase the measurement error of the sensor in the fifth zone $e_{3,t}$ between $10$\,am and $2$\,pm as shown in Figure~\ref{fig:measurement_error} to simulate a failing sensor.

\begin{figure}
	\centering \small
	\newlength\breite
	\newlength\hohe
	\setlength{\breite}{.4\textwidth}
	\setlength{\hohe}{.175\textwidth}
	\input{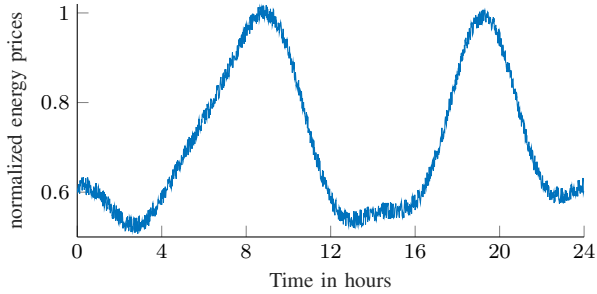}
	\vspace{-3ex}
	\caption{Energy prices $p_t$ over one day.}
	\label{fig:energy_cost}
\end{figure}

\begin{figure}
	\centering \small
	\setlength{\breite}{.4\textwidth}
	\input{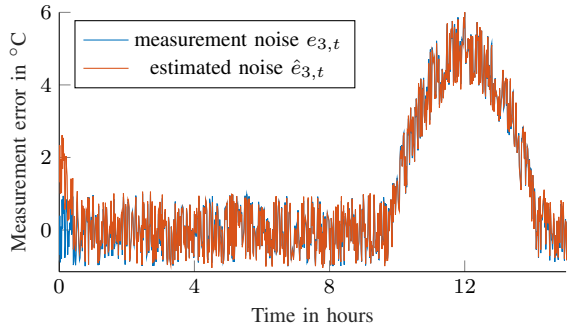}
	\vspace{-3ex}
	\caption{Real and estimated measurement error in zone~$5$.}
	\label{fig:measurement_error}
\end{figure}

The results are illustrated in Figures~\ref{fig:measurement_error}-\ref{fig:simulation_trajectories}. Figure~\ref{fig:measurement_error} shows the measurement error in the fifth zone $e_{3,t}$ and the corresponding estimate of Algorithm~1 $\hat e_{3,t}$ for the first $15$ hours. Initially, the estimate is off by $2$\,°C because of the wrong initialization, but then converges to the true measurement error in accordance with Lemma~\ref{lem:error_estimate_convergence}. Note that a slight mismatch persists due to process noise. 

Figure~\ref{fig:simulation_trajectories} shows the true closed-loop temperatures and inputs of zones~$2$~and~$5$, only one of which can be measured, together with the optimal steady state $(\eta,\theta)$ for both zones. Even though the temperature in zone~2 cannot be measured and the algorithm has to cope with process and measurement noise, the closed loop closely tracks the optimal steady state. This is true for sudden changes, i.e., the change of $\Lambda_t$ at $6$\,am and the change in $T^{set}$ at $9$\,am, as well as for gradual changes due to the fluctuation of the energy prices $p_t$. Note that the increase in measurement noise around $12$\,am has no influence on the control performance. The noise in the true temperatures is due to process noise.

\begin{figure}
	\centering \small
	\setlength{\breite}{.4\textwidth}
	\setlength{\hohe}{.8\textwidth}
	\input{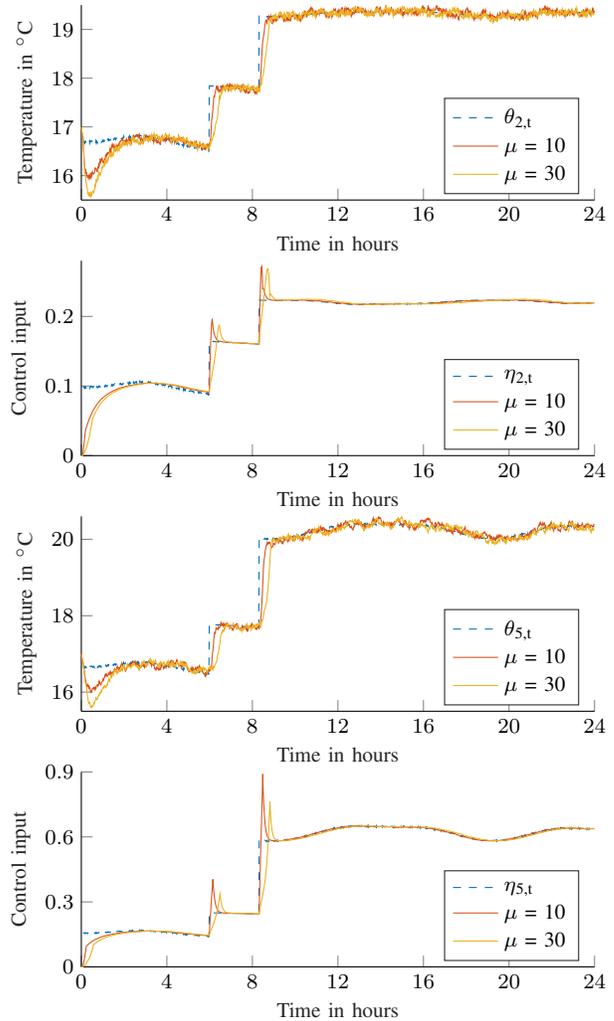}
	\vspace{-3ex}
	\caption{Optimal steady state (blue, dashed), closed-loop real temperatures and inputs for $\mu=10$ (red) and $\mu=30$ (yellow). From top to bottom: Temperatures in zone~2; Control inputs in zone~2; Temperatures in zone~5; Control inputs in zone~5.}
	\label{fig:simulation_trajectories}
\end{figure}

Comparing the different values for the prediction horizon $\mu\in\{10,30\}$, Figure~\ref{fig:simulation_trajectories} indicates that a shorter prediction horizon yields a more aggressive controller. The accumulated cost over the whole day are approximately $7165$ for \mbox{$\mu = 10$} and $7247$ for $\mu = 30$ for the same noise realization. Thus, a shorter prediction horizon yields (slightly) superior performance in this example.

\subsection{Comparison to related work}

In a second experiment, we compare Algorithm~1 to the method proposed in~\cite{Bianchin2021} for a similar setting (compare the discussion in the Introduction). In order to achieve satisfactory performance for both algorithms, we have to reduce the measurement error and sample it uniformly from the interval $[-0.1,0.1]$. For Algorithm~1, we choose the same parameters and initialization as before and choose $\mu = 10$. For the algorithm proposed in~\cite{Bianchin2021}, we set the step size $\eta$ to $0.005$. The results are illustrated in Figure~\ref{fig:simulation_comparison}. It can be seen that both algorithms are able to track the time-varying optimal steady state closely. Moreover, for these parameters, both algorithms achieve almost the same closed-loop cost. However, the algorithm proposed in~\cite{Bianchin2021} does so with a higher overshoot, more oscillations, and larger control inputs.

\begin{figure}
	\centering \small
	\setlength{\breite}{.4\textwidth}
	\input{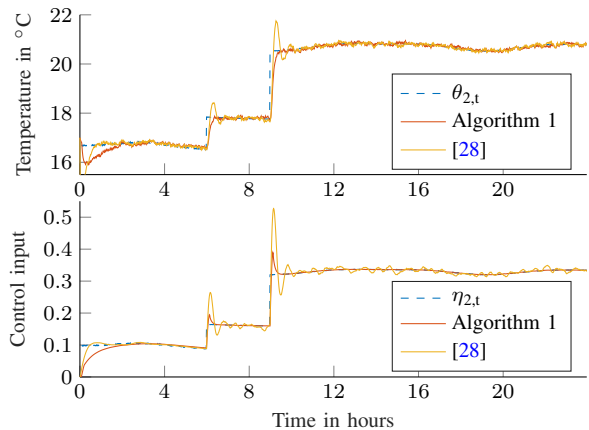}
	\vspace{-4ex}
	\caption{Optimal steady state (blue, dashed), closed-loop real temperatures and inputs for Algorithm~1 (red) and the algorithm proposed in~\cite{Bianchin2021} (yellow). From top to bottom: Temperatures in zone~2; Control inputs in zone~2}
	\label{fig:simulation_comparison}
	\vspace{-2ex}
\end{figure}

\section{CONCLUSION} \label{sec:conclusion}
In this paper, we proposed a data-driven OCO-based scheme for controlling linear dynamical systems subject to measurement noise. We only use a single persistently exciting data trajectory instead of a model of the system and output feedback to derive the control algorithm. The control scheme achieves a similar sublinear regret bound as comparable algorithms from the literature, despite only having access to noisy measurements. In particular, we show that adding measurement noise to the control problem only leads to an additional \textit{constant} term in the regret bound. Compared to previous work, the proposed algorithm is able to handle the more general and practically important case of \textit{economic} cost functions and additionally allows to relax previous assumptions on the steady-state manifold of the system.

Future work includes obtaining theoretical guarantees for the case of both process and measurement noise, as well as considering noisy a priori data. Furthermore, enabling the proposed algorithm to handle state and input constraints, which has already been achieved in a model-based setting, is an interesting direction of future research.
\section*{APPENDIX}
\subsection{Proof of Theorem~\ref{thm:regret_bound}}
Before we prove the regret bound, we first derive some auxiliary results. Note that $\alpha_t$ and $\beta_t$ in \eqref{eq:defalpha_algo} and \eqref{eq:defbeta_algo} always have a solution by Lemma~\ref{lem:feasibility}. Then, by \eqref{eq:defalpha_algo}, we have
\begin{align*}
	&\begin{bmatrix} U^{1:n} \\ Y^{1:n} \end{bmatrix} \alpha_t = \begin{bmatrix} U^{2:n} \alpha_{t-1} \\ u_{t-1} \\ Y^{2:n} \alpha_{t-1} \\ \tilde y_{t-1} - \hat e_{t-1} \end{bmatrix} \refeq{\eqref{eq:defbeta_algo}} \begin{bmatrix} U^{2:n} (\alpha_{t-1}+\beta_{t-1}) \\ u_{t-1} \\ Y^{2:n} (\alpha_{t-1}+\beta_{t-1}) \\ \tilde y_{t-1} - \hat e_{t-1} \end{bmatrix} \\
	\refeq{\eqref{eq:defhate},\eqref{eq:OutputAlgo}} &\begin{bmatrix} U^{2:n+1} \\ Y^{2:n+1} \end{bmatrix} (\alpha_{t-1} + \beta_{t-1}). \numberthis\label{eq:PR_same_init}
\end{align*}
Moreover, the input sequence generated by $\alpha_t$ is given by
\begin{align*}
	&U^{n+1:2n+\mu} \alpha_t \refeq{\eqref{eq:defalpha_algo}} \begin{bmatrix} \sigma \hat u_{t-1} \\ 1_{n} \otimes u^s_{t-1} \end{bmatrix} \\
	\refeq{\eqref{eq:PredInputs}} \, &\begin{bmatrix} \sigma(U^{n+1:n+\mu+1} (\alpha_{t-1} + \beta_{t-1})) \\ 1_{n} \otimes u^s_{t-1} \end{bmatrix} \\
	\refeq{\eqref{eq:defbeta_algo}} \, &U^{n+2:2n+\mu+1} (\alpha_{t-1} + \beta_{t-1}). \numberthis\label{eq:PR_same_input}
\end{align*}
Hence, $\alpha_t$ and $\alpha_{t-1} + \beta_{t-1}$ give rise to the same initialization \eqref{eq:PR_same_init} and the same input sequence \eqref{eq:PR_same_input} and, therefore, must produce the same output trajectory \cite{Markovsky2008}
\[
	Y^{n+1:2n+\mu} \alpha_{t} = Y^{n+2:2n+\mu+1} (\alpha_{t-1} + \beta_{t-1}) \numberthis \label{eq:PredRec}.
\]
Note that 
\begin{align*}
	Y^{n+\mu+1:2n+\mu} (\alpha_t + \beta_t) &\refeq{\eqref{eq:defbeta_algo}} 1_n\otimes y^s_t, \\
	U^{n+\mu+1:2n+\mu+1} (\alpha_t + \beta_t) &\refeq{\eqref{eq:defbeta_algo}} 1_{n+1}\otimes u^s_t,
\end{align*}
which implies that the predicted system is at the equilibrium $(u^s_t,y^s_t)$ for $n$ time steps and at the $(n+1)$-th time step, $u^s_t$ is applied again. Hence, the system remains at the same equilibrium and we have
\begin{equation}
	Y^{n+\mu+1:2n+\mu+1} (\alpha_t + \beta_t) = 1_{n+1}\otimes y^s_t. \label{eq:terminal_states}
\end{equation}
Moreover, we need the following key result on the convergence rate of projected gradient descent from \mbox{\cite[Theorem~2.2.14]{Nesterov2018}}. Let $L(z)$ be an $\alpha_z$-strongly convex and $l_z$-smooth function to be minimized. Then, one projected gradient descent step $z_1 = \Pi_Z(z_0 - \gamma \nabla L(z_0))$, where $\Pi_Z(\cdot)$ denotes projection onto the set $Z$ and $\gamma\leq\frac{2}{\alpha_z+l_z}$ is the step size, satisfies
\begin{align}
	\norm{z_1 - z_0} \leq \kappa \norm{z_0 - z^*}, \label{eq:gradient_descent}
\end{align}
where $z^* = \arg\min_{z\in Z} L(z)$ and $\kappa = 1-\alpha_z\gamma$.

Now, we are ready to bound the regret $\mathcal R$ of Algorithm~1. By definition of the regret and Lipschitz continuity of the cost functions, we have
\begin{align*}
	\mathcal R &\refeq{\eqref{eq:def_regret}} \sum_{t=0}^T L_t(z_t) - L_t(\zeta_t)
	= C_\mu + \sum_{t=\mu}^T L_t(z_t) - L_t(\zeta_t) \\
	&\leq C_\mu + L_z \sum_{t=0}^{T-\mu} \norm{z_{t+\mu} - \zeta_{t+\mu}}
\end{align*}
where $C_\mu = \sum_{t=0}^{\mu-1} L_t(z_t) - L_t(\zeta_t)$ is a constant which is independent of $T$. Applying the triangle inequality yields
\begin{align*}
	\mathcal R \leq C_\mu + &L_z \sum_{t=0}^{T-\mu} \norm{z_{t+\mu} - \hat z^\mu_t} +L_z \sum_{t=0}^{T-\mu} \norm{\hat z_t^\mu - \zeta_{t-1}} \\ + &L_z \sum_{t=0}^{T-\mu} \norm{\zeta_{t+\mu} - \zeta_{t-1}}.
\end{align*}
Again applying the triangle inequality, we get
\begin{align*}
	&\sum_{t=0}^{T-\mu} \norm{\zeta_{t+\mu} - \zeta_{t-1}} = \sum_{t=0}^{T-\mu} \norm{\sum_{i=0}^\mu \zeta_{t+i} - \zeta_{t+i-1}} \\
	\leq &\sum_{t=0}^{T-\mu} \sum_{i=0}^\mu \norm{\zeta_{t+i} - \zeta_{t+i-1}} 
	\leq (\mu+1) \sum_{t=0}^{T} \norm{\zeta_t - \zeta_{t-1}},
\end{align*}
which implies
\begin{align}
	\mathcal R \leq C_\mu + &L_z \sum_{t=0}^{T-\mu} \norm{z_{t+\mu} - \hat z^\mu_t} +L_z \sum_{t=0}^{T-\mu} \norm{\hat z_t^\mu - \zeta_{t-1}} \nonumber \\ + &L_z (\mu+1) \sum_{t=0}^{T} \norm{\zeta_t - \zeta_{t-1}}.\label{eq:preliminary_bound}
\end{align}
Next, we will establish bounds for the first two sums in \eqref{eq:preliminary_bound} separately. First, we bound the predicted regret $\sum_{t=0}^{T-\mu} \norm{\hat z_t^\mu - \zeta_{t-1}}$. We note that $\zeta_t$ is only defined for $t \in \mathbb I_{[0,T]}$. Therefore, we are free to choose $\zeta_{-1} = z^s_{-1}$. Thus, we have for each $\tau \in \mathbb I_{[0,T]}$
\begin{align*}
	&\sum_{t=0}^\tau \norm{\hat z_t^\mu - \zeta_{t-1}} \refeq{\eqref{eq:mu_ahead_prediction}} \sum_{t=0}^\tau \norm{\begin{bmatrix} u^s_{t-1} \\ Y^{n+\mu+1}\alpha_t \end{bmatrix} - \zeta_{t-1}} \\
	\refeq{\eqref{eq:PredRec}} &\sum_{t=0}^\tau \norm{\begin{bmatrix} u^s_{t-1} \\ Y^{n+\mu+2}(\alpha_{t-1} + \beta_{t-1}) \end{bmatrix}-\zeta_{t-1}} \\
	\refeq{\eqref{eq:defbeta_algo}} &\sum_{t=0}^\tau \norm{\begin{bmatrix} u^s_{t-1} \\ y^s_{t-1} \end{bmatrix}-\zeta_{t-1}} \\
	\leq &\sum_{t=0}^{\tau-1} \norm{z^s_t - \zeta_{t-1}} + \sum_{t=0}^{\tau-1} \norm{\zeta_t - \zeta_{t-1}} \\
	\refleq{\eqref{eq:OGD},\eqref{eq:gradient_descent}} &\kappa \sum_{t=0}^\tau \norm{ \hat z^\mu_t - \zeta_{t-1}} + \sum_{t=0}^{\tau-1} \norm{\zeta_t - \zeta_{t-1}}.
\end{align*}
Since $\kappa < 1$, rearranging yields
\begin{equation}
	\sum_{t=0}^\tau \norm{\hat z_t^\mu - \zeta_{t-1}} \leq \frac{1}{1-\kappa} \sum_{t=0}^{\tau-1} \norm{\zeta_t - \zeta_{t-1}}. \label{eq:hatz-zeta}
\end{equation}

Having established a bound on the second sum in~\eqref{eq:preliminary_bound}, we proceed to bound the accumulated prediction error $\sum_{t=0}^{T-\mu} \norm{z_{t+\mu} - \hat z^\mu_t}$. In order to do so, we first need a bound on $\sum_{t=0}^{T} \norm{g_t}$, where $g_t$ is defined in~\eqref{eq:defgt}. First, note that
\begin{align*}
	\sum_{t=0}^T \norm{z^s_t - \zeta_{t-1}} \refleq{\eqref{eq:OGD},\eqref{eq:gradient_descent}} &\kappa \sum_{t=0}^T \norm{\hat z^\mu_t - \zeta_{t-1}} \\
	\refleq{\eqref{eq:hatz-zeta}} ~ &\frac{\kappa}{1-\kappa} \sum_{t=0}^{T-1} \norm{\zeta_t - \zeta_{t-1}}, \numberthis \label{eq:zs-zeta}
\end{align*}
where we used~\eqref{eq:hatz-zeta} with $\tau = T$ in the last line. Therefore,
\begin{align*}
	\sum_{t=0}^T \norm{z^s_t - z^s_{t-1}} &\leq \sum_{t=0}^T \norm{z^s_t - \zeta_{t-1}} + \sum_{t=0}^T \norm{z^s_{t-1} - \zeta_{t-1}} \\
	&\leq 2 \sum_{t=0}^T \norm{z^s_{t} - \zeta_{t-1}} + \sum_{t=0}^{T-1} \norm{\zeta_t - \zeta_{t-1}} \\
	&\refleq{\eqref{eq:zs-zeta}} \frac{1+\kappa}{1-\kappa} \sum_{t=0}^{T-1} \norm{\zeta_t - \zeta_{t-1}}, \numberthis \label{eq:zs-zs}
\end{align*}
where we used $\zeta_{-1} = z^s_{-1}$, positivity of the norm and the triangle inequality in the second inequality. Then, we have

\begin{align*}
	\sum_{t=0}^{T} \norm{g_t} &\refeq{\eqref{eq:defgt}} \sum_{t=0}^{T} \norm{\begin{bmatrix} 1_{n+1} \otimes u^s_t - U^{n+\mu+1:2n+\mu+1} \alpha_t \\ 1_n \otimes y^s_t - Y^{n+\mu+1:2n+\mu} \alpha_t \end{bmatrix}} \\
	&\refeq{\eqref{eq:defalpha_algo},\eqref{eq:PredRec},\eqref{eq:terminal_states}} \sum_{t=0}^{T} \norm{\begin{bmatrix} 1_{n+1} \otimes (u^s_t - u^s_{t-1}) \\ 1_n \otimes (y^s_t - y^s_{t-1}) \end{bmatrix}}
\end{align*}
Rearranging the vector on the right-hand side yields
\begin{align*}
	\sum_{t=0}^{T} \norm{g_t} &\leq \sum_{t=0}^{T} \norm{ 1_{n+1} \otimes \begin{bmatrix} u^s_t - u^s_{t-1} \\ y_t^s - y^s_{t-1} \end{bmatrix}} \\
	&\leq \sqrt{n+1} \sum_{t=0}^{T} \norm{z^s_t - z^s_{t-1}} \\
	&\refleq{\eqref{eq:zs-zs}} \sqrt{n+1} \frac{1+\kappa}{1-\kappa} \sum_{t=0}^{T-1} \norm{\zeta_t - \zeta_{t-1}}. \numberthis \label{eq:gt}
\end{align*}
Let $\tilde Q = \left( I_{N-2n-\mu} - \left( Q \left( I_{N-2n-\mu} - H_\beta^\dagger H_\beta \right)\right)^\dagger Q \right) H_\beta^\dagger$, then we have $\beta_t = \tilde Q g_t$ by~\eqref{eq:defbeta}. Hence,
\begin{equation}
	\sum_{t=0}^{T} \norm{\beta_t} \refleq{\eqref{eq:gt}} C_\beta \sum_{t=0}^{T-1} \norm{\zeta_t - \zeta_{t-1}}, \label{eq:bound_beta}
\end{equation}
where $C_\beta = \norm{\tilde Q} \sqrt{n+1} \frac{1+\kappa}{1-\kappa}$.

Having established a bound on $\sum_{t=0}^{T} \norm{\beta_t}$, we proceed to first bound the output prediction error $\sum_{t=0}^{T-\mu} \norm{Y^{n+\mu+1} \alpha_t - y_{t+\mu}}$, and then the full prediction error $\sum_{t=0}^{T-\mu} \norm{z_{t+\mu} - \hat z^\mu_t}$. To this extent, we first bound the error of the measurement noise estimates in \eqref{eq:defhate}. Let
\[S_O = \begin{bmatrix} C^\top & (CA)^\top & \dots & (CA^{n-1})^\top \end{bmatrix}^\top\]
be the system's observability matrix and recall the definition of $\epsilon_t \refeq{\eqref{eq:def_epsilon}} \alpha_t - \alpha^*_t$ in the proof of Lemma~\ref{lem:error_estimate_convergence}.
Since $U\epsilon_t = 0$, $Y^{1:n} \epsilon_t$ describes a trajectory of the unforced system by Theorem~\ref{thm:fundamental_lemma}. Thus, by observability there exists a unique internal state $x^\epsilon_t$ such that
\[Y^{1:n}\epsilon_t = S_O x^\epsilon_{t-n}\]
holds for all $t$. Hence,
\[ x^\epsilon_{t-n} = S_O^\dagger Y^{1:n} \epsilon_t.\]
Recalling that the error trajectory follows the unforced system dynamics by Lemma~\ref{lem:error_estimate_convergence}, we can conclude
\begin{align*}
	Y^{1:n} \epsilon_t &= S_O x^\epsilon_{t-n} = S_O A x^\epsilon_{t-n-1} \\
	&= S_O A S_O^\dagger Y^{1:n} \epsilon_{t-1}.
\end{align*}
Using these arguments repeatedly, we get
\begin{align*}
	&\sum_{t=0}^{T-\mu} \norm{e_{[t-n:t-1]} - \hat  e_{[t-n:t-1]}} \refeq{\eqref{eq:estimate_recursion}} \sum_{t=0}^{T-\mu} \norm{Y^{1:n}\epsilon_t} \\
	= &\sum_{t=0}^{T-\mu} \norm{S_O A S_O^\dagger Y^{1:n} \epsilon_{t-1}}
	= \sum_{t=0}^{T-\mu} \norm{S_O A^t S_O^\dagger Y^{1:n} \epsilon_0} \\
	\leq &\norm{S_O}\norm{S_O^\dagger} \sum_{t=0}^{T-\mu} \norm{A^t} \norm{ e_{[-n:-1]} - \hat e_{[-n:-1]}}.
\end{align*}
Note that $\norm{S_O}\norm{S_O^\dagger} = \frac{\sigma_{max}(S_O)}{\sigma_{min}(S_O)}$ where $\sigma_{max}(S_O)$, $\sigma_{min}(S_O)$ denote the largest and smallest singular value of $S_O$, respectively. Moreover, since $A$ is Schur stable by Assumption~\ref{assump:system_properties}, there exist constants $c > 0$ and $\lambda \in (0,1)$ such that $\norm{A^t} \leq c\lambda^t$. Thus,
\begin{align}
	\sum_{t=0}^{T-\mu} \norm{e_{[t-n:t-1]} - \hat e_{[t-n:t-1]}} &\leq \frac{\sigma_{max}(S_O)}{\sigma_{min}(S_O)} E_0 c \sum_{t=0}^{T-\mu} \lambda^t \nonumber \\
	&\leq \frac{\sigma_{max}(S_O)}{\sigma_{min}(S_O)} \frac{c}{1-\lambda} E_0. \label{eq:e-hate}
\end{align}
Next, let 
\begin{equation} \label{eq:def_baralpha*}
	\bar \alpha^*_t = H_\alpha^\dagger \begin{bmatrix} u_{[t-n:t+\mu]} \\ 1_n \otimes u^s_{t-1} \\ y_{[t-n:t-1]} \end{bmatrix}.
\end{equation}
Then, we have $y_{t+\mu} = Y^{n+\mu+1} \bar \alpha_t^*$. Moreover, we have that
\begin{align*}
	u_{t+j} &\refeq{\eqref{eq:OutputAlgo}} U^{n+1} (\alpha_{t+j} + \beta_{t+j}) \\
	&\refeq{\eqref{eq:defalpha_algo},\eqref{eq:PredInputs}} U^{n+2} (\alpha_{t+j-1} + \beta_{t+j-1}) + U^{n+1} \beta_{t+j}.
\end{align*}
Applying~\eqref{eq:defalpha_algo}~and~\eqref{eq:PredInputs} repeatedly, we get
\begin{equation} \label{eq:ut+j_1}
	u_{t+j} = U^{n+j+1} \alpha_t + \sum_{i=0}^{j} U^{n+i+1} \beta_{t+j-i}
\end{equation}
for $0 \leq j \leq \mu-1$ and
\begin{equation}
	u_{t+\mu} = u^s_{t-1} + \sum_{i=0}^{\mu} U^{n+i+1} \beta_{t+\mu-i}. \label{eq:ut+j_2}
\end{equation}

Define $C_1 = \norm{Y^{n+\mu+1} H_\alpha^\dagger}$. Combining the above results, we are now ready to bound the output prediction error $\sum_{t=0}^{T-\mu} \norm{Y^{n+\mu+1} \alpha_t - y_{t+\mu}}$. By Theorem~\ref{thm:fundamental_lemma}, any $\alpha_t$ satisfying \eqref{eq:defalpha_algo} results in the same output $Y^{n+\mu+1}\alpha_t$ since the vector on the right-hand side of \eqref{eq:defalpha_algo} uniquely specifies the input sequence and initial condition (compare \cite{Markovsky2008}). Hence, in the following we assume without loss of generality that $\alpha_t$ is chosen according to~\eqref{eq:defalpha}.
\begin{align*}
	&\sum_{t=0}^{T-\mu} \norm{Y^{n+\mu+1} \alpha_t - y_{t+\mu}} = \sum_{t=0}^{T-\mu} \norm{Y^{n+\mu+1} (\alpha_t - \bar\alpha_t^*)} \\
	\refleq{\eqref{eq:defalpha},\eqref{eq:def_baralpha*}} &C_1 \sum_{t=0}^{T-\mu} \norm{ \begin{bmatrix} u_{[t-n:t-1]} - u_{[t-n:t-1]} \\ U^{n+1} \alpha_t - u_t \\ \vdots \\ U^{n+\mu+1} \alpha_t - u_{t+\mu} \\ 1_n \otimes (u^s_{t-1} - u^s_{t-1}) \\ \tilde y_{[t-n:t-1]} - \hat e_{[t-n:t-1]} + y_{[t-n:t-1]} \end{bmatrix} } \\[5pt]
	\refleq{\eqref{eq:ut+j_1},\eqref{eq:ut+j_2}} &C_1 \sum_{t=0}^{T-\mu} \norm{ \begin{bmatrix} U^{n+1} \beta_t \\ \sum_{i=0}^1 U^{n+i+1} \beta_{t+1-i} \\ \vdots \\ \sum_{i=0}^\mu U^{n+i+1} \beta_{t+\mu-i} \\ e_{[t-n:t-1]} - \hat e_{[t-n:t-1]} \end{bmatrix} }
\end{align*}
Define $C_e = C_1 \frac{\sigma_{max}(S_O)}{\sigma_{min}(S_O)} \frac{c}{1-\lambda}$. Then, we get
\begin{align*}
	&\sum_{t=0}^{T-\mu} \norm{Y^{n+\mu+1} \alpha_t - y_{t+\mu}} \\
	\leq &C_1 \sum_{t=0}^{T-\mu} \sum_{i=0}^\mu \norm{ U^{n+1:n+1+i} \beta_{t+\mu-i}} \\ &\quad + C_1 \sum_{t=0}^{T-\mu} \norm{e_{[t-n:t-1]} - \hat e_{[t-n:t-1]}} \\
	\refleq{\eqref{eq:e-hate}} &C_1 \sum_{t=0}^{T-\mu} \sum_{i=0}^{\mu} \norm{U^{n+1:n+\mu+1}} \norm{\beta_{t+\mu-i}} + C_e E_0 \\
	\leq &C_1 \norm{U^{n+1:n+\mu+1}} (\mu+1) \sum_{t=0}^{T} \norm{\beta_{t}}  + C_e E_0 \\
	\refleq{\eqref{eq:bound_beta}} &C_2 \sum_{t=0}^{T-1} \norm{\zeta_t - \zeta_{t-1}} + C_e E_0, \numberthis \label{eq:bound_output_pred_error}
\end{align*}
where $C_2 = C_1 \norm{U^{n+1:n+\mu+1}} (\mu+1) C_\beta$. 

Next, we are finally ready to bound the prediction error $\sum_{t=0}^{T-\mu} \norm{\hat z^\mu_t - z_{t+\mu}}$:
\begin{align*}
	&\sum_{t=0}^{T-\mu} \norm{\hat z^\mu_t - z_{t+\mu}} \refeq{\eqref{eq:mu_ahead_prediction}} \sum_{t=0}^{T-\mu} \norm{\begin{bmatrix} u^s_{t-1} - u_{t+\mu} \\ Y^{n+\mu+1} \alpha_t - y_{t+\mu} \end{bmatrix} } \\
	\refleq{\eqref{eq:ut+j_2}} &\sum_{t=0}^{T-\mu} \norm{\sum_{i=0}^\mu U^{n+i+1} \beta_{t+\mu-i}} + \sum_{t=0}^{T-\mu} \norm{Y^{n+\mu+1} \alpha_t - y_{t+\mu}}
\end{align*}
Positivity of the norm and inserting \eqref{eq:bound_output_pred_error} yields
\begin{align*}
	&\sum_{t=0}^{T-\mu} \norm{\hat z^\mu_t - z_{t+\mu}}
	\leq \norm{U^{n+1:n+\mu+1} } \sum_{t=0}^{T-\mu} \sum_{i=0}^\mu \norm{\beta_{t+\mu-i}} \\ &\quad + C_2 \sum_{t=0}^{T-1} \norm{\zeta_t - \zeta_{t-1}} + C_e E_0 \\
	&\leq \norm{U^{n+1:n+\mu+1} } (\mu+1) \sum_{t=0}^{T} \norm{\beta_{t}} \\ &\quad + C_2 \sum_{t=0}^{T-1} \norm{\zeta_t - \zeta_{t-1}} + C_e E_0 \\
	&\refleq{\eqref{eq:bound_beta}} \norm{U^{n+1:n+\mu+1} } (\mu+1) C_\beta \sum_{t=0}^{T-1} \norm{\zeta_t - \zeta_{t-1}} \\ &\quad + C_2 \sum_{t=0}^{T-1} \norm{\zeta_t - \zeta_{t-1}} + C_e E_0. \numberthis \label{eq:hatz-z}
\end{align*}
The result then follows from inserting~\eqref{eq:hatz-zeta} with $\tau = T-\mu$ and~\eqref{eq:hatz-z} into~\eqref{eq:preliminary_bound}. \endproof

\bibliographystyle{IEEEtran}
\bibliography{IEEEabrv,bib}

\begin{IEEEbiography}[{\includegraphics[width=2.5cm,height=3.2cm,clip,keepaspectratio]{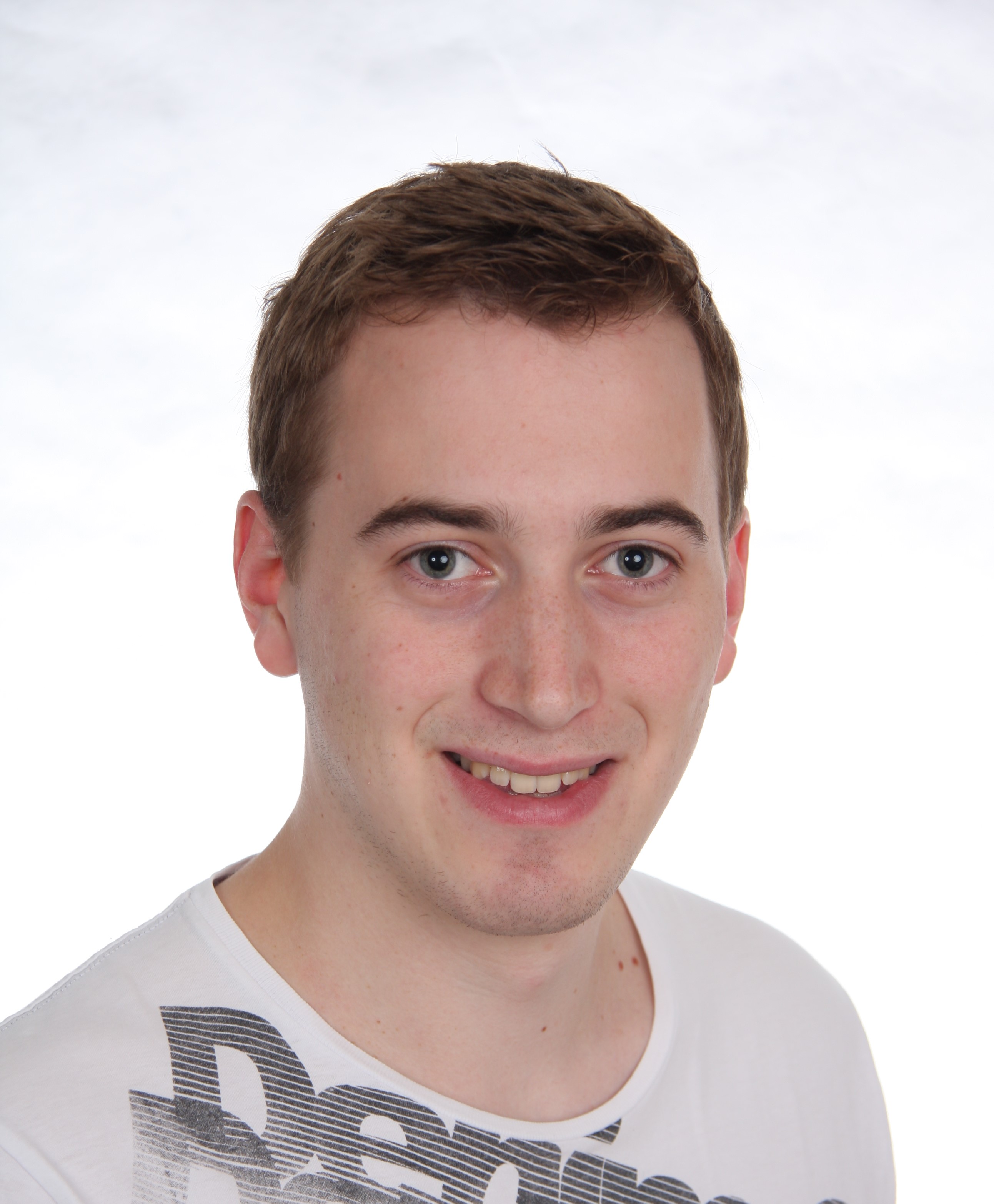}}]{Marko Nonhoff}{\space} received his Master degree in engineering cybernetics from the University of Stuttgart, Germany, in 2018. Since then, he has been a Research Assistant at the Leibniz University Hannover, Germany, where he is working on his Ph.D. under the supervision of Prof. Matthias A. Müller. His research interests are in the area of learning-based control and online optimization.
\end{IEEEbiography}

\begin{IEEEbiography}[{\includegraphics[width=2.5cm,height=3.2cm,clip,keepaspectratio]{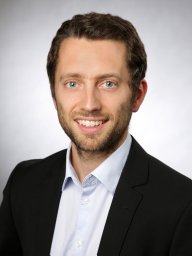}}]{Matthias A. Müller}{\space}(Senior Member, IEEE) received a Diploma degree in Engineering Cybernetics from the University of Stuttgart, Germany, and an M.S. in electrical and computer Engineering from the University of Illinois at Urbana-Champaign, US, both in 2009. In 2014, he obtained a Ph.D. in mechanical engineering, also from the University of Stuttgart, Germany, for which he received the 2015 European Ph.D.	award on control for complex and heterogeneous systems. Since 2019, he is Director of the Institute of Automatic Control and full professor at the Leibniz University Hannover, Germany. 
	
He obtained an ERC Starting Grant in 2020 and is recipient of the inaugural Brockett-Willems Outstanding Paper Award for the best paper published in Systems \& Control Letters in the period 2014-2018. His research interests include nonlinear control and estimation, model predictive control, and data-/learning-based control, with application in different fields including biomedical engineering. 
\end{IEEEbiography}

\end{document}